\documentclass[colorlinks=false]{article}

\usepackage[utf8]{inputenc}

\usepackage{graphicx} 

\usepackage{tikz}   
\usetikzlibrary{shapes,arrows,positioning}
\usepackage{tkz-graph}

\usepackage{amsmath, amssymb, dsfont}
\usepackage[a4paper,left=3cm,right=3cm,top=4cm,bottom=4cm,bindingoffset=5mm]{geometry}

\usepackage{pgfpages, pgfplots} 
\pgfplotsset{compat=1.18} 
\pgfplotsset{every tick label/.append style={font=\footnotesize}}

\usepackage{hyperref}
\hypersetup{colorlinks=false}

\usepackage[minbibnames=1,maxbibnames=4,style=alphabetic, sorting=nyt,giveninits=true]{biblatex}
\usepackage{csquotes}
\renewcommand{\P}{\mathds P}
\newcommand{\EW}{\mathds E}

\newcommand{\N}{\mathds N}
\newcommand{\R}{\mathds R}

\newcommand{\sigF}{\mathcal F}
\newcommand{\sigA}{\mathcal A}

\DeclareMathOperator{\diag}{diag}

\newcommand{\given}{\,|\,}

\newcommand{\vone}{\mathds1} 

\RequirePackage{amsthm}
\newtheoremstyle{styAufgabe} 
  {5pt}   
  {10pt}   
  {}   
  {}   
  {\normalfont\bfseries\sffamily}   
  {\medskip}   
  {\newline}   
  {\thmname{#1} \thmnumber{#2} \thmnote{(#3)}\ \hrulefill}  

\makeatletter
\def\@endtheorem{\medskip\textbf{\hrule}}
\makeatother

\theoremstyle{styAufgabe}

\newtheorem{theo}{Theorem}[section]
\newtheorem{cor}{Corollary}[section]

\newtheorem{prop}{Proposition}[section]
\newtheorem{defin}{Definition}[section]
\newtheorem{rem}{Remark}[section]
\newtheorem{rem*}{Remark}

\newcommand{\trG}[4][1]{
    \tikz[scale=#1]{ 
        \draw[fill = #2, line width = 0.2pt] (-30:0.1) circle (1.5pt); 
        \draw[fill = #3, line width = 0.2pt] ( 90:0.1) circle (1.5pt); 
        \draw[fill = #4, line width = 0.2pt] (210:0.1) circle (1.5pt); 
    }
}

\newcommand{\twoG}[3][1]{
    \tikz[scale=#1]{
        \draw[black,fill=#2, line width = 0.2pt] (0:0.1) circle (1.5pt);
        \draw[black,fill=#3, line width = 0.2pt] (180:0.1) circle (1.5pt);
    }
}

\newcommand{\configs}{\mathcal X}

\title{Fixation probability\\ in Moran-like Processes on graphs}
\author{Dr. Peter Keller\footnote{Research Assistant at the University of Potsdam, mail: peter.keller-at-uni-potsdam.de}\\
Mert U\u gurlu\footnote{Master's Student of Mathematics, University of Potsdam}}
\date{Last update: \today}

\begin{document}
\maketitle

\begin{abstract}
The well-known \emph{Isothermal Theorem} was introduced in a Nature Communications article in 2005 and has since contributed to the creation of the rich field of \emph{evolutionary graph theory}. The theorem states under which conditions certain Moran-like processes on graphs (``spatial Moran Processes'') have the same fixation probability as the classic one-dimensional Moran Process that was introduced by Moran in 1958.

Unfortunately, the Isothermal Theorem has never been proven completely. The main argument, that the projection of the process on the graph dynamics onto a one-dimensional process is a Birth-and-Death-Process, is not true in general, as the projection does not need to be Markovian.

The aim of this paper is to present a more general version of the Isothermal Theorem using martingale techniques and a generalised framework using matrix notation.

We follow up with a short study of small population size that shows the set of spatial Moran Processes with Moran fixation probability is even richer than previously understood. We underline the role played by the initial condition, and how individuals of the population are chosen for procreation.
\bigskip

\end{abstract}
\pagebreak 

\section{Introduction}

\subsection{The classic Moran Process}

In population genetics, the classic \emph{Moran Process} is a fundamental integer-valued stochastic process to understand genetic drift in finite, well mixed populations of $n\in\N$ individuals with overlapping generations and selection, and two types of competing alleles (gene variants). It was introduced in 1958, see \cite{Moran58}. There are two types of alleles: the original allele (\emph{wildtype}) and a mutation (\emph{mutant}). The Moran Process tracks the number of mutated alleles through time, where wildtypes may be replaced by mutants through asexual reproduction and vice versa. It is thus a simple example of a Birth-and-Death-Process $(M_k)_k$ on the integers $\{0,1,\ldots,n\}$ with absorbing boundaries. 

Initially, $i\in\N$ \emph{mutants} invade the population and displace wildtypes, such that the population size remains constant. After this initial invasion, no other mutants immigrate into the population. Selection of individuals (mutant or wildtype) for procreation is skewed according to the fitness parameter $r\in\R^+$ of the mutants. Fitness does not depend on the individuals. For $0<r<1$, the procreation of mutants is suppressed, for $r=1$ the mutation is neutral (no preference of mutants over wildtype), and for $r>1$, the procreation of mutants is amplified. Given there are $j$ mutants present in the population, a mutant is selected with probability
    \begin{equation}\label{eq:selprobm}
        \frac{r\frac{j}{n}}{r\frac{j}{n}+\frac{n-j}{n}}
        =\frac{1}{1+(r-1)\frac{j}{n}}\frac{r\,j}{n},
    \end{equation}
        
and a wildtype with probability
    \begin{equation}\label{eq:selprobw}
        \frac{\frac{n-j}{n}}{r\frac{j}{n}+\frac{n-j}{n}}
        =\frac{1}{1+(r-1)\frac{j}{n}}\frac{n-j}{n},
    \end{equation}
where $j$ is the number of mutants currently present in the population. The individual to be replaced is chosen uniformly at random. In some models, the selection procedure is reversed, see \cite{Diaz21} for an overview of the variants.

We denote with $p^+(j)$ the probability to increase the mutant population by one, assuming $j$ mutants present in the population. Analogously, for the probability to decrease the number of mutants by one, which we denote with $p^-(j)$. We thus denote
    \begin{equation}
    \label{eq:ClassicMoran}
    p^+(j):=\frac{r}{1+(r-1)\frac{j}{n}}\,\frac{j}{n}\frac{(n-j)}{n}, \qquad
    p^-(j):=\frac{1}{1+(r-1)\frac{j}{n}}\frac{(n-j)}{n}\frac{j}{n},
    \end{equation}
where $j$ is the current number of mutants in the population. The  states $0$ and $n$ are absorbing. The graph in figure \ref{fig:ClassMoranTransitionGraph} illustrates the dynamics.

The stopping time  
\[T_{fix}:=\inf\{k\geq 0:M_k=n\},\]
is the first time the population consists only of mutants and is called \emph{fixation time}. Since fixation is not certain, it is interesting to compute the fixation probability, i.e. $\P(T_{fix}<\infty|M_0=i)$. We shortly recall a computation using a martingale argument.

\begin{figure}
\begin{center}
    \begin{tikzpicture}[scale=1.25]
        \tikzset{every node/.style = {draw, circle, minimum size = .8cm, font = \scriptsize}};
        \tikzset{every path/.style = {->, line width = 1pt, outer sep = 2 pt}};

        \node (0)    at (-0.5, 0) {0};
        \node[draw=none] (ddd1) at ( 2, 0) {$\ldots$};
        \node (j)    at ( 5, 0) {j};
        \node[draw=none] (ddd2) at (8, 0) {$\ldots$};
        \node (n)    at (10.5, 0) {n};

        \draw[->] (ddd1) to node[draw=none, fill=white] {$p^-(1)$} (0);

        \draw[->] (ddd2) to node[draw=none, fill=white] {$p^+(n-1)$} (n);
        
        \draw[->] (j) to[bend right] node[draw=none, fill= white] {$p^-(j)$} (ddd1);
        \draw[->] (ddd2) to[bend right] node[draw=none, fill= white] {$p^-(j+1)$} (j);
        \draw[->] (ddd1) to[bend right] node[draw=none, fill= white] {$p^+(j-1)$} (j);
        \draw[->] (j) to[bend right] node[draw=none, fill= white] {$p^+(j)$} (ddd2);

        \draw (0) to[loop, out =  70, in = 110, looseness=5] node[draw=none, above, outer sep = 2pt] {1} (0);
        \draw (j) to[loop, out =  70, in = 110, looseness=5] node[draw=none, above=-25pt] {$1-p^+(j)-p^-(j)$} (j);
        \draw (n) to[loop, out =  290, in = 250, looseness=5] node[draw=none, below, outer sep = 2pt] {1} (n);
        
    \end{tikzpicture}
    \[p^+(j)=\begin{cases}
        \frac{r}{1+(r-1)\frac{j}{n}}\frac{j}{n}\frac{n-j}{n},& \text{if $0<j<n$},\\
        0,&\text{else.}
        \end{cases},\qquad
        p^-(j)=\begin{cases}
        \frac{1}{1+(r-1)\frac{j}{n}}\frac{n-j}{n}\frac{j}{n},& \text{if $0<j<n$},\\
        0,&\text{else.}
        \end{cases}\]
\end{center}
    \caption{Transition graph of the dynamics of the classic Moran Process for population size $n$.}
    \label{fig:ClassMoranTransitionGraph}
\end{figure}

\subsubsection{Symmetric case $r=1$ (neutral selection)} 
For $r=0$ the equations (\ref{eq:selprobm}) and (\ref{eq:selprobw}) simplify for all $j$, $0<j<n$, to
\[
    p^+(j)
    =\frac{j}{n}\frac{n-j}{n}
    =p^-(j).
\]
Thus, one finds that $(M_k)_{k\geq 0}$ is a martingale with respect to its canonical filtration $(\sigF_k)_k$. Using Markov Property gives
\begin{align*}
    \EW[M_{k+1}\given \sigF_k]
    &=M_k+\EW[M_{k+1}-M_k\given M_k]\\
    &=M_k+(+1)\cdot p^+(M_k)+(-1)\cdot p^-(M_k)
    =M_k.
\end{align*}
The absorption time 
    \[T:=\inf\{k\geq 0:M_k\in\{0,n\}\}\]
is almost surely finite, i.e. $\P(T<\infty\given M_0=i)=1$, for all $0<i<n$. Applying the Optional Stopping Theorem, one gets
\[
    i
    =\EW[M_0]
    =\EW[M_T]
    =(1-\P(T_{fix}<\infty))\cdot 0+\P(T_{fix}<\infty)\cdot n.
\]
Rearranging gives
\[
    \P(T_{fix}<\infty\given M_0=i)
    =\frac{i}{n}.
\]
\subsubsection{Asymmetric case $r\neq1$ (non-neutral selection)} 
For $r\neq 1$, one notices that equations (\ref{eq:selprobm}),(\ref{eq:selprobw}) fulfil for every $j$, $0<j<n$,
\[
    \frac{p^-}{p^+}
    =\frac{1}{r}.
\]
Furthermore, the function $h(j):=r^{-j}$ is harmonic on the transient states $\{1,2,\ldots,n-1\}$ and thus $(r^{-M_k})_{k\geq 0}$ is a martingale, too. The Optional Stopping Theorem then implies 
\begin{align*}
    r^i
    &=\EW[r^{M_0}]
    =\EW[r^{M_T}]
    =(1-\P(T_{fix}<\infty))\cdot r^0+\P(T_{fix}<\infty)\cdot r^n.
\end{align*}
Rearranging gives
    \[\P(T_{fix}<\infty\given M_0=i)
        =\frac{1-1/r^i}{1-1/r^n}.\]

\subsubsection{Moran fixation probabilities}
In the literature on \emph{evolutionary graph theory}, see \cite{Lieberman2005,Nowak2006,Diaz21}, the absorption probabilities
\begin{equation}
    \label{eq:MoranFixationProb}
    \varrho_i:=
    \begin{cases}
        \frac{i}{n}&r=1\\
        \frac{1-1/r^i}{1-1/r^n}&r\neq 1,
    \end{cases}
    \qquad i\in\{1,2,\ldots,n-1\},
\end{equation}
are referred to as \emph{Moran fixation probabilities}. The main idea of evolutionary graph theory is to determine under which population structures the fixation probabilities are still $(\varrho_i)_{0<i<n}$.


\subsection{Extension of the Moran Process to structured populations}
In the past, the classic Moran Process has been extended to take population structure into account, see \cite{Diaz21} for a review on the history and recent research on the topic. The population structure is now modelled via a directed, weighted graph. The vertices of the graph can either be occupied by a mutant or a wildtype. Once an individual is chosen, only its direct neighbour can be replaced with a copy.

The dynamics defines a Markov Chain, where its state space consists of all the possible marked graphs with respect to the original (unmarked) graph (details in section 2). Again, one can ask about the probability of fixation, i.e. the probability that eventually all vertices of the graph are occupied by mutants.

In \cite{Lieberman2005}, the authors claim that under some conditions on the graph and the weights, this process has the same fixation probabilities $(\varrho_i)_i$ as the Moran Process, see (\ref{eq:MoranFixationProb}). This result is known in the literature as \emph{Isothermal Theorem}. The proof uses a projection onto a simpler process $(M_k)_k$. We choose the same notation as in the previous subsection for the classic Moran process, as this projection shares many properties with the original Moran Process. They then use results from Birth-and-Death-Processes for the computation of the fixation probability. Unfortunately, this method fails, as the projection of a Markov Chain is not necessarily Markov, unless under lumpability, see \cite{Marin2017,Buchholz1994,KemenySnell1976} for extensive analysis of the conditions for a projection onto a smaller state space to be Markov. This shortcoming has not been noticed in the literature so far, not even in the recent review paper \cite{Diaz21}. Therefore, the isothermal theorem remains unproven till today. 

Attempts towards applications of martingales in the isothermal case for the microscopic spatial Moran Process have been made before in \cite{Monk14,Monk21,Monk22}. However, some of their assumptions are either not applicable (for instance independence of increase/decrease steps) or are special cases when the macroscopic spatial Moran Process is Markov, i.e. when the dynamics is lumpable.

The scope of this paper is as follows: We first generalise the model, by allowing a non-uniform selection policy for the choice of the next individual to procreate, instead of choosing uniformly at random. Then, with an analogue approach as for the classic Moran Process above, we construct a martingale $(r^{-M_k})_k$ which, regardless of the non-Markovianity of $(M_k)_k$, delivers Moran fixation probability by Optional Stopping Theorem. The Isothermal Theorem then follows as a simple corollary.


\subsection{Organisation of the paper}
In Section 2 we introduce the \emph{spatial Moran Process}, extending previous definitions such that selection of individuals for procreation is not necessarily uniform, but \emph{stationary} with respect to the weights of the graph. We state the Isothermal Theorem and then give explicit expressions for the probability to increase/decrease the mutant population in matrix notation. In section 3, we prove that the spatial Moran Process shows the same fixation probability as the simple, one-dimensional classic Moran Process, independently of the initial condition. Section 4 explores all possible choices of parameters to retain Moran fixation in the simple case of $n=2$. In particular, the example shows, that stationary selection is not necessary for Moran fixation probabilities to appear.

\section{The spatial Moran Model}

\begin{figure}
    \begin{center}
        \begin{tikzpicture}
            \node (g0) at (0,3.4) {$\configs_0$};
            \node (g1) at (3,3.4) {$\configs_1$};
            \node (g2) at (6,3.4) {$\configs_2$};
            \node (g3) at (9,3.4) {$\configs_3$};

            \node (g3) at (11.35,-3) {$\configs$};
            
        \tikzset{
            every path/.style = {->, line width = 1pt, outer sep = 4pt, draw},
            every node/.style = {draw = black, circle, fill = white}
        };
              \draw[rounded corners, dashed] (-2,-3.5) rectangle ( 11,4);
              \draw[rounded corners, fill = lightgray!50, draw = gray] (-1,-3) rectangle ( 1,3);
              \draw[rounded corners, fill = lightgray!50, draw = gray] ( 2,-3) rectangle ( 4,3);
              \draw[rounded corners, fill = lightgray!50, draw = gray] ( 5,-3) rectangle ( 7,3);
              \draw[rounded corners, fill = lightgray!50, draw = gray] ( 8,-3) rectangle (10,3);
              
              \node (000) at (0,0) { \trG[2]{white}{white}{white} }; 
              \node (001) at (3,2) { \trG[2]{white}{white}{black} };
              \node (010) at (3,0) { \trG[2]{white}{black}{white} };
              \node (100) at (3,-2) { \trG[2]{black}{white}{white} };
              \node (011) at (6,2) { \trG[2]{white}{black}{black} };
              \node (101) at (6,0) { \trG[2]{black}{white}{black} };
              \node (110) at (6,-2) { \trG[2]{black}{black}{white} };
              \node (111) at (9,0) { \trG[2]{black}{black}{black} };
              
              \path (100) -- (000);
              \path (010) -- (000);
              \path (001) -- (000);

              \path (110) -- (111);
              \path (101) -- (111);
              \path (011) -- (111);

          \tikzset{
            every path/.style = {<->, line width = 1pt, outer sep = 4pt, draw}};
                
              \path (100) -- (110);
              \path (100) -- (101);

              \path (010) -- (110);
              \path (010) -- (011);

              \path (001) -- (101);
              \path (001) -- (011);
              
        \end{tikzpicture}
    \end{center}
    \caption{General setup of the Microscopic Spatial Moran Process (MicSPM), here for $n=3$. We omit the vertices and transition probabilities, and show only occupation by 0s or 1s. The underlying graph is the complete graph $K_{3}$. Whether or not all drawn vertices exist, depends on the choice for the weights given by weight matrix $W$.}
    \label{fig:Modeln3}
\end{figure}

\subsection{Population structure}
To model the structure of a population of $n$ individuals, let $G=(E,V)$ be a directed graph, where the set of vertices is $V:=\{1,2,\ldots,n\}$. 
The $n\times n$ stochastic matrix $W:=(W(v,\tilde v))_{v,\tilde v\in V}$ is the weight matrix of the graph $G$. If $W(v,\tilde v)>0$ it means there is an edge $(v,\tilde v)\in E$, if $W(v,\tilde v)=0$ there is no edge. It is possible to reconstruct $G$ from this information alone. Also, $W(v,.)$ is a probability measure on the direct neighbours of $v$ (including $v$). 

Let further be $\vone:=(1,1,\ldots,1)$. The notation ``$x^\top$'' denotes transposition of a row vector $x$ or a matrix. We further assume that $G$ is strongly connected, i.e. for every two distinct vertices $v,\tilde v\in V$, there is a finite sequence of strictly positively weighted edges connecting $v$ with $\tilde v$ and vice versa. Therefore, there exists a unique, strictly positive left-eigenvalue $\pi:=(\pi_1,\pi_2,\ldots,\pi_n)$ that fulfils
    \begin{equation}
        \pi W=\pi,\qquad \pi\vone^\top=1.
    \end{equation}
Due to the normalisation, we can think of $\pi$ as a probability distribution on the vertices of $G$. Similarly, we can interpret the matrix $\diag(\pi)W$ as a probability measure on the edges of the complete graph $K_n$ with $n$ edges (including loops), where $\diag(\pi)$ denotes the diagonal matrix that has the entries of the vector $pi$ at its diagonal.

Moreover, we want to label a vertex with either $0$ for a wildtype, or with $1$ for a mutant. Let the collection of all possible labels of $G$ with $0$'s and $1$'s be $\configs:=\{0,1\}^V$. An element $x\in\configs$ is called \emph{graph configuration} or just configuration. Obviously, there are $2^n$ such configurations in $\configs$. We decompose $\configs$ disjointly into
    \begin{equation}
        \configs=\bigsqcup_{\ell=0}^n\configs_\ell,
        \label{eq:canonicalDecomposition}
    \end{equation}
where $\configs_\ell$ contains all configurations with exactly $\ell$ mutants resp. 1's. We call this the \emph{canonical decomposition}.

Given the graph structure and a fixed numbering of the vertices, we can represent any graph configuration $x\in\configs$ by a row vector $x\in\{0,1\}^n$ of zeros and ones, i.e. we set the $j^{th}$ component of $x$ to zero resp. one, if vertex $j$ is occupied by a wildtype resp. mutant. For convenience, we will use the notation $x$ for both the configuration as well as the vector representation. Analogously, we will use $\configs$ in both interpretations interchangeably.

\subsection{Dynamics of the microscopic spatial Moran Process}
We introduce two stochastic processes, to describe the development of the population. One describes the dynamics on the marked graph $G$, as described in the previous subsection, in a \emph{microscopic} fashion. The other is the projection onto the canonical decomposition $\configs$, that has similarities to the classic Moran Process, but does not constitute a Markov Chain in general.

\begin{defin}[Microscopic Spatial Moran Process]\label{def:msp}
Let $G$ be a directed, strongly connected graph with weight matrix $W$ and let $\pi$ be its unique strictly positive and normed left-eigenvector. We call $\pi$ \emph{selection policy}.
Let $r\in\R^+$ be a positive constant (called \emph{mutant fitness}). Let $(X_k)_{k\geq 0}$ be a Markov Chain with state space $\configs$ on the filtrated probability space $(\configs, \sigA, \sigF,\P)$, where $\sigF$ is the canonical filtration. 

For $x\in \configs$, we denote with $\zeta:=x\pi^\top$, the probability to select any vertex labelled 1 under selection policy $\pi$ in configuration $x$. If $x=X_k$ we simply write $\zeta_k$. The dynamics of $(X_k)_k$ is given by the following update rule:
\begin{itemize}
    \item 
    $X_0$ is chosen according to the initial probability distribution $\alpha$ on $\configs\setminus(\configs_0\cup\configs_N)$.
    \item 
    At time $k>0$,
    \begin{itemize}
        \item 
        a vertex $v$ is selected for procreation under selection policy $\pi$ and fitness $r$ with probability 
        \[\begin{cases}
            \frac{r\pi(v)}{1+(r-1)\zeta_k} & \text{if $v$ is labelled with 1 (mutant),}\\
            \frac{\pi(v)}{1+(r-1)\zeta_k} & \text{if $v$ is labelled with 0 (wildtype).}
        \end{cases}\]
        \item 
        Independently, another vertex $v'$ is chosen randomly according to the probability measure $W(v,.)$.
    \end{itemize}
\end{itemize}
Each update is independent of the previous one, resp. depends only on the last configuration. This defines a Glauber-dynamics. 

The two configurations where all vertices are occupied by the same type are absorbing. Further define 
\begin{equation}
    T_{fix}:=\inf\{k\geq 0: X_k=\vone\}.
\end{equation}
$T_{fix}$ describes
the first time when all vertices of $G$ are labelled with 1 resp. when all vertices are occupied by mutants.
We denote the probability that fixation is reached given $x\in\configs$
\begin{equation}
\rho_x:=\P(T_{fix}<+\infty\given X_0=x).
\end{equation}

Finally, we introduce the notation
\[\langle (X_k)_{k\geq 0},W,\pi,r\rangle\]
for the \emph{Microscopic Spatial Moran Process} (MicSMP).
\end{defin}
\bigskip

To coarsen the process, we only count the number of mutants currently present in the population (similar to the classic Moran Process). We thereby lose any information about which vertices the mutants occupy. This corresponds to projecting the Markov Chain $(X_k)_k$ onto the canonical decomposition (\ref{eq:canonicalDecomposition}).

\begin{defin}[Macroscopic Spatial Moran Process]
Let $\langle(X_k)_k,W, \pi,r\rangle$ be a given micSMP. We cal the associated projection $(M_k)_k$ defined by
    \begin{equation}
    M_k:=X_k\vone^\top\in\{0,1,2,\ldots,n\},k\geq0
    \label{eq:defMk}
    \end{equation}
the \emph{Macroscopic Spatial Moran Process} (macSMP).
\end{defin}
\bigskip

Let $\sigF$ be the canonical filtration defined by $(X_k)_k$. Note that $M_k$ is a $\sigF$-measurable, bounded function of $X_k$. Thus by Markov Property
    \[\EW[M_{k+1}\given \sigF_k]=\EW[M_{k+1}\given X_k].\]
Nevertheless, it does \emph{not} follow that $(M_k)_k$ is a Markov Chain in general, since usually
    \[\EW[M_{k+1}\given \sigF_k]\neq \EW[M_{k+1}\given M_k].\]

\subsection{The Isothermal Theorem}
In \cite{Lieberman2005}, the authors define the \emph{isothermal property} of the weight matrix $W$ in the following way: If for all $v\in V$ 
    \[\sum_{\tilde v\in V}W( v,\tilde v)=\sum_{\tilde v\in V}W(\tilde v,v),\]
then $W$ is called isothermal. The isothermal property indeed just says that for any vertex of the graph the sum of incoming weights is equal to the sum of outgoing weights. In matrix notation, this is just 
    \[W\vone^\top=W^\top\vone^\top=\vone^\top.\]
The last equality holds because of the assumed stochasticity of $W$. From this, $W$ is isothermal if and only if $W$ is bistochastic, and the invariant measure of $W$, fulfilling $\pi W=\pi$, must be uniform.

The next theorem is precise a reformulation of the isothermal theorem by Liebermann, Nowak and Hauert as given in \cite{Lieberman2005} using our notation.
\begin{theo}
    \label{theo:IsothermalTheorem}
    Let $\langle(X_k)_k, W,\pi,r\rangle$ be a micSMP with the strong assumption that $W$ is a bistochastic weight matrix and assume uniform selection. Then for all $i$, $0< i <n$, and any initial distribution $\alpha$ on a transient configuration with $i$ mutants, $0<i<n$,  
        \[\P(T_{fix}<\infty\given X_0\sim \alpha)
        =\begin{cases}
            \frac{i}{n} & r=1,\\
            \frac{1-1/r^i}{1-1/r} & r\neq 1.
        \end{cases}.\]
\end{theo}
\bigskip

We skip the proof of Theorem \ref{theo:IsothermalTheorem}, as it is a special case of Theorem \ref{theo:stationarySelectionTheorem}.

\subsection{Probability to increase/decrease the mutant population under stationary selection}
The probability to increase/decrease the mutant population by one for any weight matrix $W$, given a configuration $x\in\configs$, can be expressed in matrix notation. We write $p^+(x)$ for the (total) probability of increasing the mutant population by one, and $p^-(x)$ for the (total) probability to decrease the mutant population by one. These probabilities are transitions from a specific $x\in \configs_j$ into the set $\configs_{j\pm 1}$. 

Let $W_\pi:= \diag(\pi)W$ and $\zeta=x\pi^\top$, then by stochasticity of $W$
    \[
        x W_\pi\vone^\top
            =x\diag(\pi)\vone^\top
            =x\pi^\top
            =\zeta.
    \]
Thus
    \begin{equation}\label{eq:p+}
        p^+(x)
        :=\frac{r}{1+(r-1)\zeta}x W_\pi(\vone-x)^\top
        =\frac{r}{1+(r-1)\zeta}(\zeta-xW_\pi x^\top).
    \end{equation}


Similarly, we compute the probability to decrease the mutant population by one
    \begin{equation}\label{eq:p-}
        p^-(x)
        :=\frac{1}{1+(r-1)\zeta}(\vone-x)W_\pi x^\top
        =\frac{1}{1+(r-1)\zeta}(\zeta-x W_\pi x^\top).
    \end{equation}
The probability not to change the size of the mutant population is 
    \begin{align*}
        p_0(x)
        &:=1-p^+(x)-p^-(x)
        =1-\frac{r+1}{1+(r-1)\zeta}(\zeta-xW_\pi x^\top).
    \end{align*}

\subsubsection*{The classic Moran Process is a projection of a micSMP on the complete graph}
The spatial model introduced above is fully consistent with the classic modelling of the Moran Process as introduced in \cite{Moran58}. The homogeneous population structure of the Moran Process can be thought of as the complete graph with loops on $n$ vertices. The weight of each edge is 1/n to reflect that any individual is neighbour to any other individual including itself and is selected with uniform probability. If we thus set $W=n^{-1}\vone^\top\vone$, then the unique left-invariant vector $\pi$ of $W$ is uniform, i.e. $\pi=n^{-1}\vone$. The probability to select a mutant is $\zeta=x\pi^\top=j/n$ for any configuration $x\in\configs_j$. Note, $j$ is the number of mutants present in the configuration $x$. Therefore, the increase/decrease probabilities do no longer depend on $x$, but only on $j$. 
Since 
    \[
    x\diag(\pi)W x^\top
    =x\pi^\top\frac{1}{n}(x\vone^\top)^\top
    =\zeta^2=\Big(\frac{j}{n}\Big)^2,\]
we have for all $x\in\configs_j$ and $0<j<n$:
\begin{align*}
    p^+(x)
    &=\frac{r}{1+(r-1)\zeta}\zeta(1-\zeta)
    =\frac{r\,\frac{j}{n}}{1+(r-1)\frac{j}{n}}\frac{n-j}{j}
    =p^+(j),\\
    p^-(x)
    &=\frac{1}{1+(r-1)\zeta}(1-\zeta)\zeta
    =\frac{\frac{j}{n}}{1+(r-1)\frac{j}{n}}\frac{n-j}{j}
    =p^-(j).
\end{align*}
Therefore the projection of this micSMP is just the classic Moran Process transitions given in (\ref{eq:ClassicMoran}). Clearly, this does not hold true anymore as soon as the selection policy is non-uniform or if it is also dependent on the initial distribution, compare section 4.



\section{Stationary selection}

The aim of this section is to show that selecting individuals via the stationary distribution $\pi$ of the stochastic weight matrix $W$ is the natural choice to retain Moran fixation probability. 

\begin{theo}[]
    \label{theo:stationarySelectionTheorem}
    Let $\langle(X_k)_{k\geq 0},W,\pi,r\rangle$ be a micSMP, as in Definition \ref{def:msp}.
    Then, for any initial distribution $\alpha$ concentrated on $\configs_j$
        \[\P(T_{fix}<\infty\given X_0\sim\alpha)=
            \begin{cases}
            \frac{i}{n}&r=1,\\
            \frac{1-1/r^i}{1-1/r^n}&r\neq 1.
    \end{cases}\]
\end{theo}
\begin{proof}
We construct a particular martingale and use the Optional Stopping Theorem. We first note that $0\leq |M_k|\leq n$ for any time, so that $(M_k)_k$ is integrable. Define 
\[\Delta_k:=M_{k+1}-M_k.\]
By construction, $\Delta_k(\omega)\in\{-1,0,1\}$. We get by the Markov Property using (\ref{eq:p+}) and (\ref{eq:p-})
\begin{align}
    \EW[\Delta_k|\sigF_k]
    &=\EW[\Delta_k|X_k]
    =p^+(X_k)-p^-(X_k)\notag\\
    &=\frac{1}{1+(r-1)\zeta_k}\Big(X_k(rW_\pi-W_\pi^\top)\vone^\top+(r-1)X_kW_\pi X_k^\top\Big).
\end{align}
We have further
    \begin{align}
        X_k(rW_\pi-W_\pi^\top)\vone^\top
        &=X_k(r\diag(\pi)W\vone^\top-(\pi W)^\top)\notag\\
        &=X_k(r\pi^\top-\pi^\top)
        =(r-1)X_k\pi^\top
        =(r-1)\zeta_k.
    \end{align}
Thus
\begin{equation}
    \EW[\Delta_k|X_k]=\frac{r-1}{1+(r-1)\zeta_k}(\zeta_k-X_kW_\pi X_k^\top).
\end{equation}
\paragraph*{Neutral selection:} If $r=1$, we have $\EW[\Delta_k|X_k]=0$. Therefore $(M_k)_k$ is a $\sigF$-martingale, i.e.
\begin{equation}
    \EW[M_{k+1}\given \sigF_k]
    =\EW[\Delta_k+M_k\given \sigF_k]=M_k.
\end{equation}
Define the $\sigF$-stopping time
    \begin{equation}
        T:=\inf\{k\geq0:M_k\in\{0,n\}\}.
    \end{equation}
Since $T$ is the absorption time in a finite Markov Chain, where all transitive states communicate with positive probability, it is almost surely finite. Now we can apply the Optional Stopping Theorem, and get $i=\EW[M_0]=\EW[M_T]$. This implies
\begin{align}
    i
    &=\EW_\alpha[M_T]
    =\P(T_{fix}<\infty\given X_0\sim\alpha)\cdot n+(1-\P(T_{fix}<\infty)\cdot 0
    =\P(T_{fix}<\infty\given X_0\sim\alpha)\cdot n,
\end{align}
and thus $\P(T_{fix}\leq\infty\given X_0\sim\alpha)=\frac{i}{n}$.

\paragraph*{Non-neutral selection:} For $r\neq 1$ we note, that for all configurations $x\in\configs\setminus\{\configs_0\cup\configs_n\}$
\begin{equation}
    \frac{p^-(x)}{p^+(x)}
    =\frac{1}{r},
\end{equation}
compare (\ref{eq:p+}-\ref{eq:p-}). Define the process
    \begin{equation}
        N_k:=r^{-M_k},\qquad k\in\N,
    \end{equation}
where $N_0=r^{-i}$ a.s. Again, $(N_k)_k$ is $\sigF$-adapted and bounded since $M_k$ is bounded, and thus integrable. The process $(N_k)_k$ is indeed a $\sigF$-martingale
\begin{align}
    \EW[r^{-M_{k+1}}\given \sigF_k]
    &=\EW[r^{-M_k}r^{-\Delta_k}\given \sigF_k]
    =r^{-M_k}\EW[r^{-\Delta_k}\given X_k]\\
    &=r^{-M_k}\Big(r^{1}p^-(X_k)+r^0\,(1-p^+(X_k)-p^-(X_k))+r^{-1}p^+(X_k)\Big)
    =r^{-M_k},
\end{align}
i.e. $\EW[N_{k+1}\given \sigF_k]=N_k$.
Applying the Optional Stopping Theorem for $T$ and $(N_k)_k$ delivers
\begin{align}
    r^{-i}
    =\EW_\alpha[N_0]
    =\EW_\alpha[N_T]
    =r^0(1-\P(T_{fix}<\infty\given X_0\sim\alpha))+r^{-n}\P(T_{fix}<\infty\given X_0\sim\alpha).
\end{align}
Rearranging gives 
\begin{equation}
    \P(T_{fix}<\infty\given X_0\sim\alpha)
    =\frac{1-1/r^i}{1-1/r^n},
\end{equation}
as desired.
\end{proof}
\bigskip

\begin{cor}
    A micSMP under stationary selection has the same fixation probability as the classic Moran Process.
\end{cor}
\bigskip

\begin{cor}
    The Isothermal Theorem \ref{theo:IsothermalTheorem} is now a direct consequence of Theorem \ref{theo:stationarySelectionTheorem}.
\end{cor}
\bigskip

\subsection*{A remark on the choice of selection policy}
Instead of selecting the next individual with the stationary probability distribution $\pi$, we can also choose with any probability distribution $\mu$ on $\{1,2,\ldots,n\}$. In light of Theorem \ref{theo:stationarySelectionTheorem}, we loose the constancy of the ratio between increase and decrease probabilities $p^+$ and $p^-$, as shown in the following theorem.

\begin{theo}[]
    \label{theo:stationarySelectionNoAlternative}
    In a micSMP $\langle (X_k)_k,W,\mu,r\rangle$ with arbitrary selection policy $\mu$, we have for any configuration $x$
    \begin{align*}
        p^+_\mu(x)
        =\frac{r}{1+(r-1)x\mu^\top}(xW_\mu\vone^\top-xW_\mu x^\top),\\
        p^-_\mu(x)
        =\frac{1}{1+(r-1)x\mu^\top}(xW_\mu^\top\vone^\top-xW_\mu x).
    \end{align*}
    Further, $p^-/p^+$ is a constant function if
    if and only if $\mu=\pi$, where $\pi W=\pi$. In that case
    \[
        \frac{p^-_\mu(x)}{p^+_\mu(x)}
        =\frac{1}{r}.
    \]
\end{theo}
\bigskip

\begin{proof}
Replacing $\pi$ in (\ref{eq:p+}-\ref{eq:p-}) with $\mu$, we get for the increase and decrease probabilities (adapting $\zeta=x\mu^\top$ and $W_\mu=\diag(\mu)W$)
\begin{align*}
    p^+_\mu(x)
    &=\frac{r}{1+(r-1)\zeta}x W_\mu (\vone-x)^\top
    =\frac{r}{1+(r-1)\zeta}(xW_\mu\vone^\top-xW_\mu x^\top),\\
    p^-_\mu(x)
    &=\frac{1}{1+(r-1)\zeta}(\vone-x)W_\mu x^\top
    =\frac{1}{1+(r-1)\zeta}(xW_\mu^\top\vone^\top-xW_\mu x^\top).
\end{align*}
Note, that with arbitrary selection policy $\mu$, we still have $X_kW_\mu\vone^\top=\zeta_k$ in $p^+_\mu$. But since $\mu$ is no longer stationary, we can not simplify $X_kW_\mu^\top\vone^\top$ in the same way in $p^-_\mu$.

Now, for arbitrary selection policy $\mu$ the ratio of decrease and increase probability can be expressed as
\begin{equation}
    \frac{p^-_\mu(X_k)}{p^+_\mu(X_k)}
    =\frac{1}{r}\left(1+\frac{X_k(W_\mu-W_\mu^\top)\vone^\top}{X_k W_\mu (\vone-X_k)^\top }\right).
    \label{eq:ratioConst}
\end{equation}
This expression is equal to $1/r$ independently of $X_k$, if and only if 
\[(W_\mu-W_\mu^\top)\vone^\top=0^\top,\]
where $0^\top$ denotes the vector that has only zeros.
This is equivalent to 
\begin{align*}
    0^\top
    &=\diag(\mu)W\vone^\top-W^\top\diag(\mu)\vone^\top
    =\mu^\top-W^\top\mu^\top,
\end{align*}
and thus becomes the condition $\mu W=\mu$ for stationarity of $\mu$ with respect to $W$. This implies that the only choice for $\mu$ is the stationary distribution of $W$ to make (\ref{eq:ratioConst}) constant. In that case
\[
    \frac{p^-_\mu(x)}{p^+_\mu(x)}=\frac{1}{r}.
\]
\end{proof}

Nevertheless, Theorem \ref{theo:stationarySelectionNoAlternative} does not exhaust all the choices of initial condition and selection policies to have Moran fixation probability $\varrho_i$ in a micSMP, see next section. 

\section{Study of fixation probability for small population size}

Theorem \ref{theo:stationarySelectionTheorem} gives sufficient assumptions for a micSMP to have Moran fixation, but it does say nothing about the necessity. To understand this at least for small population size, we compute the general fixation probability explicitly for small populations and explore parameter choices such that Moran fixation holds.

\subsection{The case $n=2$}
For two individuals, the parameter choice is limited, thus it is sufficient to define a weight matrix (with the graph)
    \begin{equation}
    W=\begin{pmatrix}
        1-w_1 & w_1\\
        w_2 & 1-w_2
    \end{pmatrix},
    \qquad
    \begin{tikzpicture}
        \tikzset{
            every node/.style={font=\sffamily\tiny, fill =white},
            every edge/.style={->,bend left = 20, draw}
        }
        
        \node (1) at (-1,0) {1};
        \node (2) at (1,0) {2};

          
        \path (1) edge node {$w_1$} (2);
        \path (2) edge node {$w_1$} (1);
        
        \draw[->] (1) to[loop, out =  160, in = 200, looseness=5] node[draw=none, left, outer sep = 2pt] {$1-w_1$} (1);

        \draw[->] (2) to[loop, out =  -20, in = 20, looseness=5] node[draw=none, right, outer sep = 2pt] {$1-w_2$} (2);
        
    \end{tikzpicture}
    \label{eq:weightMatrix}
    \end{equation}
with $w_1,w_2\in(0,1]$, initial condition $\alpha=(a,1-a)$, $a\in[0,1]$, fitness parameter $r>0$ and selection policy $\mu=(m,1-m)$, $m\in[0,1]$. We do not assume stationarity of $\mu$ with respect to $W$.

The state space of this micSMP is $\{(0,0),(0,1),(1,0),(1,1)\}=:\{\twoG{white}{white},\twoG{white}{black},\twoG{black}{white},\twoG{black}{black}\}$. With this encoding of the states, we give transition graph and transition matrix in figure \ref{fig:TransitionGraphMatrixN2}.

\begin{figure}
\begin{center}
    \begin{tikzpicture}
        \tikzset{every path/.style = {->, line width = 1pt, outer sep = 4pt, draw}};
        
        \node (00) at (-2, 0.0) {\twoG[2]{white}{white}};
        \node (01) at ( 0, 1.5) {\twoG[2]{white}{black}};
        \node (10) at ( 0,-1.5) {\twoG[2]{black}{white}};
        \node (11) at ( 2, 0.0) {\twoG[2]{black}{black}};

        \path (01) -- (00);
        \path (01) -- (11);
        \path (10) -- (00);
        \path (10) -- (11);
        
    \end{tikzpicture}
\end{center}

\begin{center}
    \begin{tabular}{c|c|cc|c}
            & \twoG{white}{white} & \twoG{white}{black} & \twoG{black}{white} & \twoG{black}{black}\\\hline
      \twoG{white}{white} & 1     & 0     & 0     & 0    \\\hline
      \twoG{white}{black} & $\frac{\mu w_1}{1+(r-1)(1-\mu)}$&$1-\frac{\mu w_1+r(1-\mu)w_2}{1+(r-1)(1-\mu)}$&0&$\frac{r(1-\mu)w_2}{1+(r-1)(1-\mu)}$                          \\
      \twoG{black}{white} & $\frac{(1-\mu)w_2}{1+(r-1)\mu}$ &0&$1-\frac{(1-\mu)w_2+r\mu w_1}{1+(r-1)\mu}$&$\frac{r\mu w_1}{1+(r-1)\mu}$                     \\\hline
      \twoG{black}{black} & 0     & 0     & 0     & 1
    \end{tabular}
\end{center}
    \caption{Transition graph and transition matrix. For simplicity, the transition probabilities are not given in the transition graph and loops are omitted.}
    \label{fig:TransitionGraphMatrixN2}
\end{figure}

The following proposition shows that stationary selection is not the only parameter choice for Moran fixation. Indeed, stationary selection implies Moran fixation for any initial distribution $\alpha$. The other solution constructs a selection policy (non-stationary), where the initial distribution $\alpha$ is fixed.

\begin{prop}\label{prop:moranfixation}
    Let a micSMP $\langle(X_k)_k,W,\mu,r \rangle$ with $W$ as defined in (\ref{eq:weightMatrix}) with $\mu=(m,1-m)$ and $\alpha=(a,1-a)$, where $a,m\in[0,1]$. Let $c:=w_1/w_2$. Then, Moran fixation holds, if
        \begin{itemize}
            \item $\mu$ is stationary with respect to $W$, independent of the choice of $a$.
            \item if $c$ and $r$ are not simultaneously equal to one and 
            \[
                m
                =\frac{a\,(r+1)-r}{a\,(r+1)(1-c)+c-r}, 
                \qquad a\in\Big[\frac{\min(1,r)}{r+1},\frac{\max(1,r)}{r+1}\Big].
            \]
        \end{itemize}
\end{prop}
\bigskip

\begin{proof}
Since the transition graph is simple, we compute the fixation probability for arbitrary parameters:
\begin{align}
    \P(T_{fix}<\infty\given X_0\sim \alpha)
    &=(a,1-a)\begin{pmatrix}
        \frac{m w_1+r(1-m)w_2}{1+(r-1)(1-m)}&0\\
        0&\frac{(1-m)w_2+r m w_1}{1+(r-1)m}
    \end{pmatrix}^{-1}
    \begin{pmatrix}
        \frac{r(1-m)w_2}{1+(r-1)(1-m)}\\
        \frac{rm w_1}{1+(r-1)m}
    \end{pmatrix}\notag\\
    &=\frac{ra\,(1-m)}{mc+r(1-m)}+\frac{rm\,(1-a)}{(1-m)\frac{1}{c}+rm}
    .\label{eq:surfaceAbsProb}
\end{align}
The first part of the proposition is a direct consequence of Theorem \ref{theo:stationarySelectionTheorem}. In this case the unique left-eigenvalue $\pi$ of $W$ is 
    \[\pi=\Big(\frac{1}{c+1},\frac{c}{c+1}\Big).\]
Plugging $m=(c+1)^{-1}$ into (\ref{eq:surfaceAbsProb}) confirms
    \begin{align*}
        \P(T_{fix}<\infty\given X_0\sim \alpha)
        &=\frac{ra\,(1-m)}{mc+r(1-m)}+\frac{rm\,(1-a)}{(1-m)\frac{1}{c}+rm}\\
        &=\frac{ra}{1+r}+\frac{r\,(1-a)}{1+r}
        =\frac{r}{r+1}
        =\frac{1-r^{-1}}{1-r^{-2}}
        =\varrho_1.
    \end{align*}
For the second part, we note that Moran fixation holds if
\[
1=\frac{\P(T_{fix}<\infty\given X_0\sim \hat\alpha)}{\varrho_1}
    =(r+1)\Big(\frac{a\,(1-m)}{mc+r\,(1-m)}+\frac{(1-a)\,m}{(1-m)/c+rm}\Big).
\]
Rearranging for $m$ gives the desired result.
\end{proof}

\begin{figure}
    \centering
    \includegraphics[width=\textwidth]{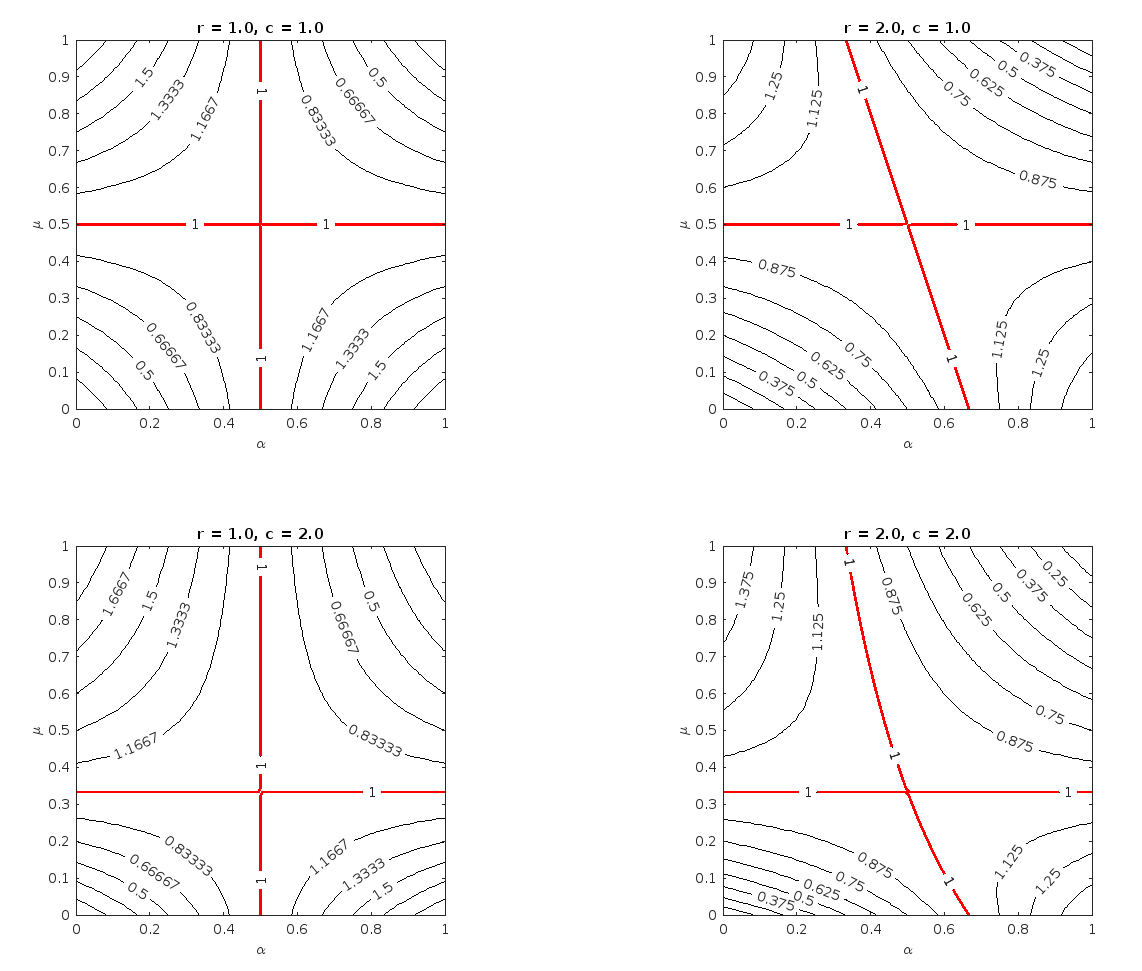}
    
    \caption{Plot of the function $F$ from Corollary \ref{cor:Fsym} parameterised by initial distribution $\alpha=(a, 1-a)$ and selection policy $\mu=(m,1-m)$. The red isolines with value 1 encode the possible choices for $(a,m)$ such that Moran fixation holds. We further chose four different value sets for the parameters $r$ and $c$.}
    \label{fig:absortpionProbs}
\end{figure}

The ratio between the general fixation probability and the Moran fixation probability shows some symmetries, which are usefull to explore the parameter space.

\begin{cor}\label{cor:Fsym}
    For a micSMP in the setting of Proposition \ref{prop:moranfixation}, we define
        \[F(m,a\given c,r):=\frac{\P(T_{fix}<\infty\given X_0\sim \hat\alpha)}{\varrho_1}.\]
    Then, $F$ has the following symmetries, i.e.
    \[F(m,a\given c,r)
    =F\Big(m, 1-a\,\Big|\, c, \frac{1}{r}\Big)
    =F\Big(1-m,a\,\Big|\,\frac{1}{c},r\Big).\]
\end{cor}
\bigskip

Thanks to the symmetries provided by Corollary \ref{cor:Fsym}, it is possible to plot a contour plot of $F$ for different parameter choices, see figure \ref{fig:absortpionProbs}.

\subsection{The case $n=3$}
The size of the transition matrix for the micSMP grows exponentially with $n$, since there are $2^n$ states. Therefore, the expressions for the fixation probabilities already for $n=3$ are not treatable in full generality, since it demands for inverting a six by six matrix. We thus investigate an example weight matrix by Galanis, see \cite{Galanis2017} p. 89, that was initially thought to be a counterexample for the Isothermal Theorem \ref{theo:IsothermalTheorem} in that the isothermal property is not necessary to have Moran fixation. See also \cite{Diaz21} for a discussion of the isothermal case. In fact, Galanis's example is also a counterexample to Theorem \ref{theo:stationarySelectionTheorem} in the sense that stationary selection is also not necessary to have Moran fixation. 


To illustrate the complexity of the situation for $n=3$, we give the full $8\times 8$ transition matrix. We denote a configuration with $\trG{white}{white}{white}$, where vertex one is on the lower right and numbering continues counterclockwise around the centre. Mutants are noted in black. For example $\trG{white}{black}{white}\in\configs_1$ indicates a mutant on vertex two, where all other vertices are occupied by wildtypes. 

The weight matrix $W$ has nine free parameters (constraint by stochasticity), i.e.
    \[
    W:=
    \begin{pmatrix}
        1-w_{12}-w_{13} & w_{12} & w_{13}\\
        w_{21} & 1-w_{12}-w_{23} & w_{23}\\
         w_{31} & w_{32} &1-w_{31}-w_{32}
    \end{pmatrix},
    \]
where $w_{ij}\in[0,1]$. Note that we still assume strong connectivity of $W$.

Finally, the transition matrix and transition graph of the micSMP for $n=3$ are given in figure \ref{fig:transitionGraphn3}. It becomes apparent that it will be impossible to explore all 13 parameters (two each for selection policy and initial distribution). In the next section we thus explore the situation for a fixed weight matrix.

\begin{rem}
There is an additional complication for any initial distribution $\alpha$ that is not concentrated on either $\configs_1$ or $\configs_2$, but possibly on both, when the Moran fixation probabilities $(\varrho_i)_{0<i<3}$ hold. The resulting fixation probability is a mixture of both Moran fixation probabilities, i.e.
    \[\P(T_{fix}<\infty\given X_0\sim \alpha)
    =\alpha_{\given\configs_1}\varrho_1+\alpha_{\given\configs_2}\varrho_2.\]
Nevertheless, it is also possible that we have Moran fixation starting from a configuration with one initial mutant, where starting from a configuration with two initial mutants does not yield Moran fixation. 
\end{rem}



\begin{figure}
\begin{center}
    \begin{tikzpicture}[scale=0.95]
        \tikzset{every path/.style = {->, line width = 1pt, outer sep = 4pt}};
        
        \node (000a) at (-30:6) { \trG[2]{white}{white}{white} };
        \node (000b) at ( 90:6) { \trG[2]{white}{white}{white} };
        \node (000c) at (210:6) { \trG[2]{white}{white}{white} };
        
        \node (100) at ( -30:3) { \trG[2]{black}{white}{white} };
        \node (110) at (  30:3) { \trG[2]{black}{black}{white} };
        \node (010) at (  90:3) { \trG[2]{white}{black}{white} };
        \node (011) at ( 150:3) { \trG[2]{white}{black}{black} };
        \node (001) at ( 210:3) { \trG[2]{white}{white}{black} };
        \node (101) at ( 270:3) { \trG[2]{black}{white}{black} };
        
        \node (111) at ( 0, 0) { \trG[2]{black}{black}{black} };

        \path[
            every node/.style={font=\sffamily\tiny, fill =white},
            every edge/.style={bend left = 15, draw}
        ]
            (010) edge node {$w_{21}$} (110)
            (010) edge node {$w_{23}$} (011)
            (001) edge node {$w_{32}$} (011)
            (001) edge node {$w_{31}$} (101)
            (100) edge node {$w_{12}$} (110)
            (100) edge node {$w_{13}$} (101)
            (110) edge node {$w_{31}$} (010)
            (110) edge node {$w_{32}$} (100)
            (110) edge[bend left=0] node {$w_{23}+w_{13}$} (111)
            (011) edge node {$w_{13}$} (010)
            (011) edge node {$w_{12}$} (001)
            (011) edge[bend left=0] node {$w_{21}+w_{31}$} (111)
            (101) edge node {$w_{21}$} (001)
            (101) edge node {$w_{23}$} (100)
            (101) edge[bend left=0] node {$w_{32}+w_{12}$} (111)
            (100) edge[bend left=0] node {$w_{21}+w_{31}$} (000a)
            (010) edge[bend left=0] node {$w_{12}+w_{32}$} (000b)
            (001) edge[bend left=0] node {$w_{13}+w_{23}$} (000c)
            ;
    \end{tikzpicture}  
\end{center}
\vspace{1cm}
\begin{center}
\begin{tabular}{c|c|ccc|ccc|c}
      & \trG{white}{white}{white} & \trG{white}{black}{white} & \trG{white}{white}{black} & \trG{black}{white}{white}& \trG{black}{black}{white}& \trG{white}{black}{black} & \trG{black}{white}{black}& \trG{black}{black}{black}\\\hline
     \trG{white}{white}{white} &1&0&0&0&0&0&0&0\\\hline
     \trG{white}{black}{white} &$\frac{w_{12}+w_{32}}{r+2}$&$\ast$&0&0&$\frac{r\,w_{21}}{r+2}$&$\frac{r\,w_{23}}{r+2}$&0&0\\
     \trG{white}{white}{black} &$\frac{w_{23}+w_{13}}{r+2}$&0&$\ast$&0&0&$\frac{r\,w_{32}}{r+2}$&$\frac{r\,w_{31}}{r+2}$&0\\
     \trG{black}{white}{white} &$\frac{w_{21}+w_{31}}{r+2}$&0&0&$\ast$&$\frac{r\,w_{12}}{r+2}$&0&$\frac{r\,w_{13}}{r+2}$&0\\\hline
     \trG{black}{black}{white} &0&$\frac{w_{31}}{2r+1}$&0&$\frac{w_{32}}{2r+1}$&$\ast$&0&0&$\frac{r\,(w_{23}+w_{13})}{2r+1}$\\
     \trG{white}{black}{black} &0&$\frac{w_{13}}{2r+1}$&$\frac{w_{12}}{2r+1}$&0&0&$\ast$&0&$\frac{r\,(w_{21}+w_{31})}{2r+1}$\\
     \trG{black}{white}{black} &0&0&$\frac{w_{21}}{2r+1}$&$\frac{w_{23}}{2r+1}$&0&0&$\ast$&$\frac{r\,(w_{32}+w_{12})}{2r+1}$\\\hline
     \trG{black}{black}{black}&0&0&0&0&0&0&0&1
\end{tabular}
\end{center}
    \caption{Transition graph and matrix of the micSMP for $n=3$. Note that the absorbing state $\trG{white}{white}{white}$ is displayed thrice to avoid overlapping transitions. Further, only the weights of $W$ are displayed in the graph to keep it simple.}
    \label{fig:transitionGraphn3}
\end{figure}

\subsection{Weight Matrix by Galanis}
The weight matrix in the counterexample by Galanis is 
\begin{equation}\label{eq:GalanisW}
W=
\begin{pmatrix}
    0 & 1/4 & 3/4\\
    1/4 & 0 & 3/4\\
    1/2 & 1/2 & 0
\end{pmatrix},
\end{equation}
which has stationary distribution $\pi=(2/7,2/7,3/7)$. Note that under stationary distribution $\pi$, states 1 and 2 are strongly lumpable, see \cite{Buchholz1994}. One could interpret this as the property that the model is not able to distinguish whether a mutant resides at site 1 or 2. 

Due to the high complexity, we only look at the neutral selection case $r=1$, in which most of the expressions simplify significantly. 

\begin{prop}\label{prop:Moranfixationn3neutral}
    Let $\langle (X_k)_k, W, \mu,1\rangle$ be a micSMP with weight matrix as defined in (\ref{eq:GalanisW}) and $\mu$ arbitrary selection matrix under neutral selection. Let $\alpha=(a_1,a_2,1-a_1-a_2,0,0,0)$ (starting with one mutant a.s.) and $\mu=(m_1,m_2,1-m_1-m_2)$.

    Then, we have
        \[\P(T_{fix}<\infty\given X_0\sim\alpha)
            =\frac{2a_2 + 3m_1 - 3a_1m_1 + 3a_1m_2 - 5a_2m_1 - 2a_2m_2}{m_1 + m_2 + 2}.
        \]
    The model displays Moran fixation probability $\varrho_1=1/3$ in the following three cases:
    \begin{enumerate}
        \item $a_1=a_2=1/3$ and $\mu$ arbitrary.
        \item $m_1=2/7$, $a_1\neq1/3$ and $a_2=(9a_1-1)/6$.
        \item An implicit solution is given by:
        \[0=
        -a_1+\frac{2-5m_1-2m_2}{3(m_1-m_2)}a_2+\frac{8m_1-m_2-2}{9(m_1-m_2)},
        \]
    if $m_1\neq m_2$.
    \end{enumerate}
\end{prop}
\bigskip

The result is a computation using a CAS, so we skip the proof here. The computation of the fixation probability involves inverting a six by six matrix symbolically, similar to the ansatz in (\ref{eq:surfaceAbsProb}). The rest is a simple manipulation of the terms to try isolate parameters, which is not always possible, as the implicit third solution suggests.



The proposition is indeed only a computation for the case of a single mutant and under neutral selection. A similar computation can be done for two initial mutants, but with initial distribution $\alpha=(0,0,0,a_3,a_4,1-a_3-a_4)$, with $a_3,a_4\in[0,1]$. This would introduce another two parameters and would likely increase the dimension of the solution space.

\section{Final discussion}
The microscopic Spatial Moran Process (micSMP) as a Moran-like Process on a graph plays a pivotal role in \emph{evolutionary graph theory}, see \cite{Nowak2006}. The strategy for the proof from \cite{Lieberman2005} bases on the assumption that the projection $(M_k)_k$ on the number of mutants is a Markovian Birth-and-Death-Process, but this is not generally the case, as is clearly visible in (\ref{eq:p+}) and (\ref{eq:p-}).

With Theorem \ref{theo:stationarySelectionTheorem} we have closed the gap in the proof of \ref{theo:IsothermalTheorem} by Liebermann, Hauert and Nowak, and showed that a much larger range of models with generalised selection of individuals also retains Moran fixation probability. We may thus need to interpret Moran fixation as a consequence of selecting stationary with respect to the population structure respectively the weight matrix, instead of as a consequence of the isothermal property and uniform selection. 

Our study of the simple case of two individuals and Galanasis's counterexample for three individuals in section 4 also hints at another class of models with Moran fixation that depend on the initial distribution as well as on the selection policy, see Propositions \ref{prop:moranfixation} and \ref{prop:Moranfixationn3neutral}. We conjecture, that it is always possible to choose initial condition and selection policy such that Moran fixation occurs for $n\geq2$. Unfortunately, the general case has a lot more parameter choices than the simple case of two individuals, especially since we can have Moran fixation starting with $i$ individuals, but not with $j\neq i$. It is therefore much more difficult to access and to predict suitable parameter choices that allow Moran fixation outside of stationary selection regime. If independence of the initial distribution is required, we conjecture that Theorem \ref{theo:stationarySelectionTheorem} exhausts the class of graph-models with Moran Fixation probability completely. 


Some extensions to the model have been investigated, too. For example in \cite{Melissourgos2022}, where the authors introduce different weight matrices for mutants and wildtypes each. When we use non-uniform selection policies for each type, we find a similar result as Melissourgos et al. in their case under isothermal property. Again, the result holds, if both policies are stationary with respect to their respective weight matrices. The proof is completely analogous to Theorem \ref{theo:stationarySelectionTheorem}.

Finally, we conjecture a \emph{universality of fixation probability} as stated in \cite{Adlam2014}. The main observation is that 
    \[\frac{p^+(x)}{p^+(x)+p^-(x)}=\frac{r}{r+1},\qquad \frac{p^-(x)}{p^+(x)+p^-(x)}=\frac{1}{r+1}\]
for all $x\in\configs$ under stationary selection (in the original article only under isothermal property). Adlam and Nowak interprete this as the micSMP conditional to non-idleness, i.e. the distribution of $M_k$ under the event $\{|X_{k+1}-X_k|\neq 0\}$. Again, this neglects the nature of the resulting process, as this constitutes only a Conditional Markov Chain, but not a classic Markov Chain. The reason is that there is a condition on the immediate future and the same issue that the projection to the number of mutants is not necessarily Markov. Since conditional Markov Chains can be characterised via martingales, see \cite{Bielecki2017}, there might be interesting ways to construct other suitable martingales that give more insight into the dynamics. However, this is beyond the scope of the current project.


Finally, the authors are confident that stationary selection plays a role in further investigation of more general models, like multi-allele-models as in \cite{Etheridge2009} with parent independent mutation or immigration.

\section{Acknowledgements}
We thank Prof. Andrey Pilipenko and Prof. Sylvie R\oe lly for fruitful discussions. 
We further thank Prof. Martin Nowak for pointing us towards \cite{Adlam2014}.

\pagebreak
\printbibliography
\end{document}